\documentclass[11pt,a4]{article}
\usepackage{latexsym}
\usepackage{amsmath}
\usepackage{amssymb}
\usepackage{amsthm}
\usepackage{amscd}
\usepackage{color}

\definecolor{dblue}{rgb}{0,0,0.45}
\definecolor{red}{rgb}{0.7,0,0}

\newtheorem{theorem}{Theorem}[section]
\newtheorem{lemma}[theorem]{Lemma}
\newtheorem{corollary}[theorem]{Corollary}
\newtheorem{proposition}[theorem]{Proposition}

\theoremstyle{definition}

\newtheorem{remark}[theorem]{Remark}
\newtheorem{definition}[theorem]{Definition}

\theoremstyle{remark}

\begin{document}

\title{Density, Duality and Preduality in Grand Variable Exponent Lebesgue and Morrey Spaces}
\author{Alexander Meskhi
\footnote{Alexander Meskhi:
(a) Department of Mathematical Analysis, A. Razmadze Mathematical Institute, I. Javakhishvili Tbilisi State University, 6 Tamarashvili Str., 0177 Tbilisi,
Georgia; (b) Department of Mathematics, Faculty of Informatics and Control Systems, Georgian Technical University, 77, Kostava St., Tbilisi, Georgia \quad email:meskhi@rmi.ge, a.meskhi@gtu.ge
}
and
Yoshihiro Sawano
\footnote{
Yoshihiro Sawano:
Department of Mathematics and Information Sciences, Tokyo Metropolitan University, Minami-Ohsawa 1-1, Hachioji, Tokyo, 192-0397, Japan \quad
email:yoshihiro-sawano@celery.ocn.ne.jp
}
}

%\date{Received: date / Accepted: date}
% The correct dates will be entered by the editor

\maketitle

\begin{abstract}
In this note some structural properties of grand variable exponent Lebesgue/ Morrey spaces
over spaces of homogeneous type
are obtained. In particular, it is proved that the closure of $L^{\infty}$ and the closure of $L^{p(\cdot), \lambda(\cdot)}$
in grand variable exponent Morrey space $L^{p(\cdot), \lambda(\cdot), \theta}$ coincide
if the measure of the underlying space is finite.
Moreover, we get two different characterizations of this class. Further,
duality and preduality of grand variable exponent Lebesghue space defined on quasi-metric measure spaces
with $\sigma$-finite measure are obtained,
which is new even when $\mu$ is finite.
\end{abstract}
{\bf keywords}
Morrey spaces,
Nakano spaces,
Iwaniec-Sbordone space,
Grand variable exponent Morrey spaces,
density,
class of bounded functions,
predual space
and
dual space
\\
{\bf 2010 classification}: 46B10, 46B26

\section{Introduction}
In recent years
we realized that the classical function spaces are no longer appropriate spaces
when we attempt to consider a number of contemporary problems arising naturally in
many other branches in science such as non-linear elasticity theory, fluid mechanics,
image restoration, mathematical
modelling of various physical phenomena, solvability problems of non-linear partial differential equations
and so on.
It was evident that many such problems are naturally related to the problems with non-standard local growth
(see e.g., \cite{AMS08,RaRe,Ruzicka00,SamkoProg,Wunderli10,Zh} and references therein).
It thus became necessary to introduce and study some new function spaces from various viewpoints.
As examples of these spaces,
we can list variable exponent spaces and grand function spaces.
These spaces were intensively investigated by researchers in the last two decades.
In this note we deal with grand variable exponent Morrey spaces
(GVEMS for short)
$L^{p(\cdot), \lambda(\cdot), \theta}$, where $p(\cdot)$ and $\lambda(\cdot)$ are variable parameters.

%In recent years it was realized that the classical function spaces are no longer appropriate spaces
%when we attempt to consider a number of contemporary problems arising naturally in
%many other branches in science such as non-linear elasticity theory {\color{red}???}, fluid mechanics \cite{AnRo06,Ruzicka00}, image restoration
%\cite{AMS08,Wunderli10}, mathematical
%modelling of various physical phenomena {\color{red}???}, solvability problems of non-linear partial differential equations {\color{red}???}, etc.
%It thus became necessary to introduce and study some new function spaces from various viewpoints. These spaces are variable exponent spaces and grand %function spaces. These spaces were intensively investigated by researchers in the last two decades. In this note we deal with exponent %Morrey spaces $L^{p(\cdot), \lambda(\cdot), \theta}$, where $p(\cdot)$ and $\lambda(\cdot)$ are variable parameters.

The study of function spaces with variable exponent has been a very active area of research nowadays.
The variable exponent Lebesgue space $L^{p(\cdot)}$ is a special case of this class of functions, which was introduced by W. Orlicz in the 1930s and subsequently generalized by I. Musielak and W. Orlicz. Later H. Nakano specified it in \cite{Nakano50,Na}. For the mapping properties of operators in Harmonic Analysis in variable exponent function spaces we refer to the monographs \cite{CF,DHHR,KMRS1,KMRS2} and the survey \cite{INS} as well as references cited therein.

T. Iwaniec and C. Sbordone introduced grand Lebesgue spaces in the 1990s \cite{IwSb} when they studied integrability of Jacobian under the minimal condition. In subsequent years, quite a number of problems in Harmonic Analysis and the theory of non-linear differential equations were studied in these spaces (see, e.g., the papers \cite{Fi,FiGuJa,FiRa,KMRS1,KMRS2}).

Morrey spaces $L^{p,\lambda}$ were introduced based on the work of C. Morrey in 1938 \cite{405a},
where Morrey studied regularity problems of solutions to partial differential equations, and provided a useful tool in the regularity theory of PDE's (for Morrey spaces we refer to the books \cite{187a,kuf};
see also \cite{rafsamsam} where an overview of various generalizations may be found).

Grand Morrey spaces $L^{p), \lambda, \theta}$ were introduced by the first author in \cite{meskhi},
where
he studied the boundedness of integral operators in these spaces.

In \cite{Rafeiro} the author considered the space $L^{p), \lambda, \theta}$, where the authors \lq \lq grandified'' the parameter $\lambda$
as well. We mention also the paper \cite{OOS} for the relevant results.

In this paper we work on a quasi-metric measure space $(X, d, \mu)$.
In the first half of this paper
we assume that $\mu$ is a finite measure.
Let $L^{p(\cdot), \lambda(\cdot), \theta}$ denote grand variable exponent Morrey spaces
(shortly GVEMS) over $X$.
One of our aims is to characterize classes of functions
$\Big[ L^{\infty}\Big]_{L^{p(\cdot), \lambda(\cdot), \theta}}$
and
$\Big[ L^{p(\cdot), \lambda(\cdot)}\Big]_{L^{p(\cdot), \lambda(\cdot), \theta}}$
which are respectively defined to be the closure in $L^{p(\cdot), \lambda(\cdot), \theta}$
of classes $L^{\infty}$ and $L^{p(\cdot), \lambda(\cdot)}$.
Here $L^{p(\cdot), \lambda(\cdot)}=L^{p(\cdot), \lambda(\cdot),0}$
is the variable Morrey space which having the second variable parameter $\lambda$.
In particular we show that these classes are the same and give two different characterizations of this class.
As a special cases
we characterize $\Big[ L^{\infty}\Big]_{L^{p(\cdot), \theta}}$ and $\Big[ L^{p(\cdot)}\Big]_{L^{p(\cdot), \theta}}$
for the grand variable exponent Lebesgue space (shortly GVELS)
$L^{p(\cdot), \theta}$.
The space
$\Big[ L^{p(\cdot)}\Big]_{L^{p(\cdot), \theta}}$
was introduced in \cite{KoMeGMJ}, where the authors
studied the mapping properties of operators
of harmonic analysis in these spaces.
Further, in this paper GVELS is introduced
over a measure space with $\sigma$-finite measure
and its preduality/duality properties are investigated.

\section{Preliminaries}

Let $(X, d, \mu)$ be a quasi-metric measure space.
Recall that a quasi-metric $d$ is a function $d:X\times X\rightarrow [0,\infty)$ which satisfies the following conditions:\\

(a) $d(x,y)=0$ if and only if $x=y$.\\

(b) For all $x,y\in X$, $d(x,y)= d(y,x)$.\\

(c) There is a constant ${\mathcal{K}}>0$ such that $d(x,y)\leq {\mathcal{K}} (d(x,z)+d(z,y))$ for all $x,y,z\in X$.\\

Let
$$d_{X}:={\rm diam}(X):=\sup\{d(x,y): x,y\in X\}$$
be the diameter of $X$. We denote by $B(x,r)$ the ball with center $x$ and radius $r$, i.e.,
$B(x,r): = \{y \in X : d(x, y) < r\}$.

We assume that $\mu\big( \{x\}\big) =0$
for all $x\in X$, and that $\mu(X)< \infty$
for the time being.
A measure $\mu$ is said to satisfy the doubling condition ($\mu\in {\rm D C}(X)$)
if there is a constant $D_{\mu}>0$ such that
$$\mu(B(x,2r))\le D_{\mu}\mu(B(x,r))$$
for every $x\in X$ and $r>0$.\\

The triple $(X, d, \mu)$ is called a quasi-metric measure space. If $\mu$ is doubling, then $(X, d, \mu)$ is called a space of homogeneous type (shortly SHT).

We denote by $P(X)$ the family of all real-valued $\mu$-measurable functions $p(\cdot)$ on $X$ such that
$$ 1<p_-:= p_-(X):= \inf \{ p(x): x\in X \}$$
and that
$$
p_+:= p_+(X):= \sup \{p(x): x\in X\} < \infty.
$$

If $E$ is a measurable set of $X$, then we use the notation:
$$p_-(E):= \inf \{ p(x): x\in E \};\;\; p_+(E):= \sup \{p(x): x\in E\}. $$

Let $p(\cdot)\in P(X)$.
The Lebesgue space with variable exponent $p(\cdot)$,
denoted by $L^{p(\cdot)}(X)$ (or by $L^{p(x)}(X)$),
is the class of all measurable $\mu$-functions $f$ on $X$ for
which
\[
 S_{p}(f):= \int_{X} |f(x)|^{p(x)} d\mu(x) <\infty.
\]

The norm in $L^{p(\cdot)}(X)$ is defined as follows:
Let $f$ be a measurable function.
Then define
\[
\| f\|_{L^{p(\cdot)}}=\| f\|_{L^{p(\cdot)}(X)}= \inf \bigg\{ \lambda>0: S_p\left(\frac{f}{\lambda}\right) \leq 1 \bigg\}.
\]

It is known that $L^{p(\cdot)}$ is a Banach space (see, e.g., \cite{KoRa,Sawano-Survey}).
Further, H\"older's inequality holds in the following form:

$$ \Bigg| \int\limits_X f(x)g(x) d\mu(x)\Bigg| \leq
 \Big( 1+ \frac{1}{p_-} - \frac{1}{p_+}\Big) \|f\|_{L^{p(\cdot)}} \| g\|_{L^{p'(\cdot)}}, \;\;\; p'(\cdot)= \frac{p(\cdot)}{p(\cdot)-1}. $$

The following relation of $L^{p(\cdot)}$ spaces will be useful for us:
for all $p(\cdot),q(\cdot)\in P(X)$,

\begin{equation}\label{embedding}
\| f\|_{L^{p(\cdot)}} \leq \big(1+\mu(X)\big) \| f\|_{L^{q(\cdot)}}, \;\; f\in L^{q(\cdot)}, \;\; p(x) \leq q(x).
\end{equation}

The next property of variable exponent Lebesgue space is important;
see \cite{NoSa12}, for example.
\begin{proposition}\label{esti} Let $1< p_- \leq p_+< \infty$.
Then the following relations between the norm and modular of $L^{p(\cdot)}$ space hold:

$$ \| f\|_{L^{p(\cdot)}}^{p_+} \leq S_{p}(f) \leq \| f\|_{L^{p(\cdot)}}^{p_-}, \;\;\; \| f\|_{L^{p(\cdot)}} \leq 1;$$

$$ \| f\|_{L^{p(\cdot)}}^{p_-} \leq S_{p}(f) \leq \| f\|_{L^{p(\cdot)}}^{p_+}, \;\;\; \| f\|_{L^{p(\cdot)}} \geq 1.$$
\end{proposition}

%For other properties of spaces $L^{p(\cdot)}$ we refer to \cite{Sh, KoRa, Sa1}.

Let ${{\mathcal{P}}}_{\mu}^{\log}(X)$ be the class of those exponents $p(\cdot)$ that satisfy the following log-H\"older condition: there exists a positive constant $a$ such that for all $x,y\in X$
with $\mu(B_{x,y}) \leq 2^{-1}$,
\begin{equation}\label{WL-+}
|p(x)- p(y)| \leq \frac{a}{- \ln (\mu(B_{x,y}))},
\end{equation}
where $B_{x,y}:= B(x, d(x,y))$.

In the case when we deal with $\Omega\subset {\Bbb{R}}^n$
and $\mu$ is the Lebesgue measure over $\Omega$,
condition (\ref{WL-+}) reads as follows: there is a positive constant $d$ such that for all
$x,y\in \Omega$ with $|x-y| \leq 2^{-1}$, then
\begin{equation}\label{WL-+Om}
|p(x)- p(y)| \leq \frac{d}{- \ln |x-y|}.
\end{equation}

Let $(X, d, \mu)$ be an SHT. If $r(\cdot) \in {{\mathcal{P}}}_{\mu}^{\log}(X)$, then the following Diening \cite{Di} type estimate holds for the exponent $r(\cdot)$ (see \cite{KMRS1}, Chapter 4):
there is a positive constant $A$ such that for all balls $B\subset X$ and all $x,y\in B$, the following estimates hold:
\begin{equation}\label{Dien}
\mu(B)^{r_-(B)- r_+(B)} \leq A.
\end{equation}

\begin{equation}\label{Dien1}
\mu(B)^{|r(x)- r(y)|} \leq A.
\end{equation}

It is known that the log-H\"older continuity condition guarantees the boundedness of operators
in Harmonic Analysis over $L^{p(\cdot)}$
(see e.g. the monographs mentioned above and references cited therein).

Let $ \lambda(\cdot)$ be a $\mu$-measurable function on $X$ with values in $[0,1]$. Suppose that $p(\cdot) \in P(X)$. We define
 variable Morrey space $L^{p(\cdot),\lambda(\cdot)}$ to be a collection of those measurable $f:X\mapsto {\Bbb{R}}$ for which the norm
$$\| f\|_{ L^{p(\cdot),\lambda(\cdot)} }=
\sup_{x\in X, \; r\in (0,\ell)} \mu(B(x,r))^{\frac{-\lambda(x)}{p(x)}} \|f \chi_{B(x,r)} \|_{L^{p(\cdot)}}
$$
is finite.

Variable exponent Morrey spaces were
initially introduced in \cite{AHS} over Euclidean spaces
and later in \cite{KoMe}
over quasi-metric measure spaces.
We refer the monograph \cite{KMRS2} (Chapter 13) and the survey \cite{Sawano-Survey}
for the mapping and other properties of variable exponent Morrey and Morrey-type spaces.

It is obvious that if $\lambda\equiv 0$, then $L^{p(\cdot),\lambda(\cdot)}$
is a variable exponent Lebesgue spaces $L^{p(\cdot)}.$

In this note we are interested in the grand variable exponent Morrey space $L^{p(\cdot),\lambda(\cdot), \theta}$,
where $\lambda(\cdot)$ is a $\mu$-measurable function on $X$
with values in $[0,1]$, $p(\cdot) \in P(X)$ and $\theta>0$ is a fixed parameter.
This space is defined by the norm:
\begin{eqnarray*}
\lefteqn{
\| f\|_{L^{p(\cdot), \lambda(\cdot), \theta}}
}\\
&:=& \sup_{0<c <p_- -1} c^{ \frac{\theta}{p_- -c}}
\left(
\sup_{x\in X, \; r\in (0,\ell)} \mu(B(x,r))^{\frac{-\lambda(x)}{p(x)-c}} \| f \chi_{B(x,r)} \|_{L^{p(\cdot)-c}}
\right) \\
&=& \sup_{0<c <p_- -1} c^{ \frac{\theta}{p_- -c}} \|f\|_{L^{p(\cdot)-c, \lambda(\cdot)}},
\end{eqnarray*}
where $\|f\|_{L^{p(\cdot)-c, \lambda(\cdot)}}$ is the variable Morrey space norm of $f$.

If $p(\cdot)$ and $\lambda(\cdot)$ belong to the class ${{\mathcal{P}}}_{\mu}^{\log}(X)$,
then due to estimates (\ref{Dien}) and (\ref{Dien1})
for $p(\cdot)$ and $\lambda(\cdot)$,
the norm $\| f\|_{L^{p(\cdot), \lambda(\cdot), \theta}(X)}$
is equivalent to
\begin{align*}
\lefteqn{
\| f\|_{ L^{p(\cdot), \lambda(\cdot), \theta}}
}\\
&\approx
\sup_{0<c <p_- -1} c^{ \frac{\theta}{p_- -c}}
\left(
\sup_{x\in X, \; r\in (0,\ell)} \mu(B(x,r))^{-\frac{\lambda_-(B)}{p_- (B)-c}} \| f\chi_{B(x,r)}\|_{L^{p(\cdot)-c}}
\right).
\end{align*}

If $\lambda(\cdot)\equiv 0$, then the space
$L^{p(\cdot), \lambda(\cdot), \theta}\equiv L^{p(\cdot),0, \theta}$
coincides with the grand Lebesgue space
$L^{p(\cdot), \theta}$.

Let $\Omega$ be a bounded domain in ${\Bbb{R}}^n$.
Our spaces cover many classical spaces.
If $\lambda(\cdot)\equiv 0$, $p=p_c=$ const, then $L^{p(\cdot), \lambda(\cdot), \theta}(\Omega)$
is the grand Lebesgue space $L^{p_c),\theta}(\Omega)$ introduced in \cite{GrIwSb}.
We remark that
$L^{p(\cdot), \lambda(\cdot), 1}(\Omega)$ is the Iwaniec--Sbordone \cite{IwSb} space $L^{p_c)}(\Omega)$. The space $L^{p_c)}(\Omega)$ naturally arises, for example, when studying integrability problems of the Jacobian under minimal hypotheses (see \cite{IwSb}), while $L^{p_c),\theta}(\Omega)$ is related to the investigation of the nonhomogeneous $n$-harmonic equation:
$\text{div}\;A(x, \nabla u) = \mu$ (see \cite{GrIwSb}). It is known (see e.g., \cite{Fi}) that the space $L^{p_c),\theta}(\Omega)$ is non-reflexive and non-separable.

By using embedding \eqref{embedding} we can check that the following properties hold for $p(\cdot) \in P(X)$:
%$$ L^{p(\cdot)} \hookrightarrow L^{p(\cdot), \theta} \hookrightarrow L^{p(\cdot)-c}, \;\; 0< c<p_- -1;$$
$$
L^{p(\cdot),\lambda(\cdot)} \hookrightarrow
L^{p(\cdot),\lambda(\cdot), \theta} \hookrightarrow
L^{p(\cdot)-c,\lambda(\cdot)}, \;\; 0< c<p_- -1, \theta>0.
$$

%The following statement will be useful for us:
%{\bf Proposition B.}
%\begin{propa} \cite{KoMeGMJ} Let $p\in $\lambda(x)\in [0,1] $ is a $\mu$-measurable function on $X$ with values in $[0,1]$, $p(\cdot) \in P(X)$ and %$\theta>0$.

%{\rm (a)} The space $L^{p(\cdot), \lambda(\cdot), \theta}(X)$ is a Banach space.

%{\rm (b)} The closure of $L^{p(\cdot)}(X)$ in $L^{p(\cdot), \theta}(X)$ consists of those $f\in L^{p(\cdot), \theta}(X)$ for which %$\lim\limits_{c\to 0} \varepsilon^{\frac{\theta}{p_- -c }} \| f(\cdot)\|_{L^{p(\cdot)-c}(X)}=0$.
%\end{propa}

%Let us denote by
Having clarified the definition of the main function spaces,
we move on to the closed subspaces.
Let us denote by
$\Big[L^{p(\cdot), \lambda(\cdot)}\Big]_{ L^{p(\cdot), \lambda(\cdot), \theta}}$
and
$\Big[L^{\infty}\Big]_{ L^{p(\cdot), \lambda(\cdot), \theta} }$
the closure of $L^{p(\cdot), \lambda(\cdot)}$ and $L^{\infty}$ in $L^{p(\cdot), \lambda(\cdot), \theta}$,
respectively.

\section{Density of $L^{\infty}$ and $L^{p(\cdot), \lambda(\cdot)}$ in $L^{p(\cdot), \lambda(\cdot), \theta}$ over finite measure spaces}

The following statement can be proved in the same way
as the corresponding statement for
function spaces with constant exponent on domains in ${\Bbb{R}}^n$
(see \cite[pp. 873-874]{KMRS2}) but we give the proof for the sake of completeness.

\begin{proposition}\label{varepsilon0}
Let $(X, d, \mu)$ be a finite quasi-metric measure space and let $p(\cdot) \in P(X)$.
Suppose that $\theta>0$
and $\lambda(\cdot)\in [0,1]$.
Then
$$ \lim_{c \to 0} c^{ \frac{\theta}{p_- -c}} \|f\|_{L^{p(\cdot)-c, \lambda(\cdot)}}=0 $$
 for $f\in \Big[L^{p(\cdot), \lambda(\cdot)}\Big]_{L^{p(\cdot), \lambda(\cdot), \theta}}$.
\end{proposition}

\begin{proof}
Let $f\in \Big[ L^{p(\cdot), \lambda(\cdot)} \Big]_{ L^{p(\cdot), \lambda(\cdot), \theta}}$
and let $\delta>0$. Then there exists $f_{n_0}\in L^{p(\cdot), \lambda(\cdot)}$ such that
$$
\| f- f_{n_0}\|_{L^{p(\cdot), \lambda(\cdot), \theta}} < \frac{\delta}{2}.
$$
Let $0<c \ll 1$.
Denote by $C_{p, \lambda, X}$
the constant independent of $c$.
For $f_{n_0}$ and $\delta$, embedding \eqref{embedding} yields that
\begin{eqnarray*}
&& c^{\frac{\theta}{p_- - c}} \sup_{x\in X, \; 0<r<d_X}
\mu(B(x,r))^{-\frac{\lambda(x)}{p(x)-c}} \|f_{n_0}\chi_{B(x,r)}\|_{L^{p(\cdot)-c}}\\
&\leq& C_{p,\lambda, X} \; c^{ \frac{\theta}{p_- -c} }
\sup_{x\in X, 0<r<d_X} \mu(B(x,r))^{-\frac{\lambda(x)}{p(x)}} \|f_{n_0}\chi_{B(x,r)}\|_{L^{p(\cdot)}}\\
&=& C_{p, \lambda, X} \; \;
c^{\frac{\theta}{p_- -c}} \| f_{n_0}\|_{L^{p(\cdot), \lambda(\cdot)} }
<
\frac{\delta}{2}.
\end{eqnarray*}
Taking now such an $f_{n_0}$, we find that
\begin{eqnarray*}
&& c^{\frac{\theta}{p_- - c}} \sup_{x\in X, \; 0<r<d_X}
\mu(B(x,r))^{-\frac{\lambda(x)}{p(x)-c}}
\|f\chi_{B(x,r)}\|_{L^{p(\cdot)-c, \theta}}\\
&\leq & c^{\frac{\theta}{p_- - c}} \sup_{x\in X, \; 0<r<d_X}
\mu(B(x,r))^{-\frac{\lambda(x)}{p(x)-c}}
\|(f-f_{n_0}) \chi_{B(x,r)}\|_{L^{p(\cdot)-c, \theta}}\\
&\quad+& c^{\frac{\theta}{p_- - c}} \sup_{x\in X, \; 0<r<d_X}
\mu(B(x,r))^{-\frac{\lambda(x)}{p(x)-c}} \|f_{n_0}\chi_{B(x,r)}\|_{L^{p(\cdot)-c, \theta}}\\
&\leq& \|f-f_{n_0}\|_{ L^{p(\cdot), \lambda(\cdot), \theta} }
+ C_{p, X, \lambda} \; c^{\frac{\theta}{p_- -c}} \| f_{n_0}\|_{L^{p(\cdot), \lambda(\cdot)}} \\
&\leq&
\frac{\delta}{2} + C_{p, \lambda, X} \; c^{\frac{\theta}{p_- -c}} \| f_{n_0}\|_{L^{p(\cdot), \lambda(\cdot)}}< \delta.
\end{eqnarray*}
Thus, we have the desired result.
\end{proof}

%The following statements were proved in \cite{Saw}:

%\begin{propa}\label{lem:151104-2}
%Let $U$ be a closed subspace enjoying the lattice property.
%Then
%$UL^{p),\tau}(\Omega,\mu)$
%has the lattice property.
%\end{propa}

The following statement for grand Lebesgue spaces with constant exponent was proved in \cite{Sawano}.

\begin{proposition}\label{Sawano-Result}
Let $(X, d, \mu)$ be a finite quasi-metric measure space, $p(\cdot) \in P(X)$,
$\lambda(\cdot) \in [0,1]$ and $\theta>0$. Then we have
\[ \Big[L^\infty\Big]_{L^{p(\cdot),\lambda(\cdot), \theta}} =
\left\{ f\in L^{p(\cdot),\lambda(\cdot), \theta}: \lim\limits_{N\to \infty}
\left\| \chi_{\{|f|>N\}} f \right \|_{L^{p(\cdot),\lambda(\cdot), \theta}}=0 \right\}. \]
\end{proposition}

\begin{proof}
We follow the proof of \cite[Lemma 3.7]{Sawano}. Let $f\in L^{p(\cdot),\lambda(\cdot), \theta}$ be such that
\[
\lim\limits_{N\to \infty}
\left\| \chi_{\{|f|>N\}}f \right\|_{L^{p(\cdot),\lambda(\cdot), \theta}}=0.
\]
For each $N\in \mathbb{N}$, define $f_N:=\chi_{\{ 	|f|\le N 	\}} f$. Then $f_N\in L^\infty$ and
\[ \|f-f_N\|_{L^{p(\cdot),\lambda(\cdot), \theta}}\to 0 \]
as $N\to \infty$. Thus, $f\in\Big[ L^\infty \Big]_{ L^{p(\cdot),\lambda(\cdot), \theta}}$.

Conversely, let $f\in \Big[ L^\infty \Big]_{L^{p(\cdot),\lambda(\cdot), \theta}}$.
For any $\delta >0$, choose
$g \in L^\infty$
such that \[ \|f-g\|_{L^{p(\cdot), \lambda(\cdot), \theta}}<\delta. \]
For every $N\in \mathbb{N}$, we have
\begin{align*}
\left|\chi_{\{|f|>N\}}f\right| \le |f-g| + \left|\chi_{\left\{|f|>N\}\cap \{|g|\le \frac{N}{2}\right\}}g\right|
+ \left|\chi_{\left\{ |g|>\frac{N}{2}\right\}}g\right|.
\end{align*}
On the set
$\{|f|>N\}\cap \{|g|\le \frac{N}{2}\}$, we have
\begin{align*}
|g|\le \frac{N}{2} <\frac{|g|}{2}
\le \frac{|f-g|}{2}+
\frac{|g|}{2},
\end{align*}
and hence, $|g|\le |f-g|$.
Therefore, if $N>2\|g\|_{L^\infty}$, we have
\begin{align}
\left\|\chi_{\{|f|>N\}}f\right\|_{L^{p(\cdot),\lambda(\cdot), \theta}}
&\le
2\left\|f-g\right\|_{L^{p(\cdot),\lambda(\cdot), \theta}}<2\delta.
\end{align}
Thus, $\lim\limits_{N \to \infty} \left\| \chi_{\{|f|>N\}} f \right\|_{L^{p(\cdot),\lambda(\cdot), \theta}}=0$.
\end{proof}

Let us define
\begin{align*}
{\widetilde{L}}^{p(\cdot), \lambda(\cdot), \theta}
&:= \{ f\in L^{p(\cdot),\lambda(\cdot), \theta}: \lim_{c \to 0}
c^{ \frac{\theta}{p_- -c}} \|f\|_{L^{p(\cdot)-c, \lambda(\cdot), \theta}}=0\},\\
{\overline{L}}^{p(\cdot), \lambda(\cdot), \theta}
&:= \left\{ f\in L^{p(\cdot),\lambda(\cdot), \theta}: \lim\limits_{N\to \infty}
\left\| \chi_{\{|f|>N\}} f \right \|_{L^{p(\cdot),\lambda(\cdot), \theta}}=0 \right\}.
\end{align*}

The main statement of this note reeds as follows:

\begin{theorem} \label{main}
Let $(X, d, \mu)$ be an SHT with $\mu(X)<\infty$,
and let $p(\cdot) \in P(X)$.
Suppose that $p(\cdot) \in {{\mathcal{P}}}_{\mu}^{\log}(X)$, $\lambda(\cdot) \in [0,1]$ and $\theta>0$. Then
\begin{equation}\label{maineqiality}
\Big[L^\infty\Big]_{L^{p(\cdot),\lambda(\cdot), \theta}}
=
\Big[L^{p(\cdot), \lambda(\cdot)}\Big]_{L^{p(\cdot),\lambda(\cdot), \theta}}.
 \end{equation}
Or equivalently,
\begin{equation}\label{coroll}
\widetilde{L}^{p(\cdot), \lambda(\cdot), \theta}
=
\overline{L}^{p(\cdot), \lambda(\cdot), \theta}.
\end{equation}
\end{theorem}

Taking $\lambda(\cdot)=0$ in Theorem \ref{main}, we have
the corresponding result for GVELS.
\begin{corollary} Let $p(\cdot) \in P(X)$ and let $(X, d, \mu)$ be an SHT
with $\mu(X)<\infty$.
Suppose that $p(\cdot)\in {{\mathcal{P}}}_{\mu}^{\log}(X)$ and $\theta>0$.
Then
\begin{equation}\label{maineqiality1}
\Big[L^\infty\Big]_{L^{p(\cdot),\theta}}
=
\Big[L^{p(\cdot)}\Big]_{L^{p(\cdot),\theta}}.
\end{equation}
Or equivalently,
\begin{equation}\label{coroll1}
\widetilde{L}^{p(\cdot), \theta}
=
\overline{L}^{p(\cdot), \theta}.
\end{equation}
\end{corollary}

{\em Proof} of Theorem \ref{main}. Since $p\in {{\mathcal{P}}}_{\mu}^{\log}(X)$
we have that \eqref{Dien} and \eqref{Dien1} hold for $r(\cdot)=1/p(\cdot)$.
Since $\mu(X)<\infty$, it can be checked that the embedding
$$\Big[L^\infty \Big]_{L^{p(\cdot),\lambda(\cdot), \theta}} \hookrightarrow \Big[ L^{p(\cdot), \lambda(\cdot)} \Big]_{L^{p(\cdot),\lambda(\cdot), \theta}} $$
is valid. Further, by Proposition \ref{varepsilon0} we have that
\begin{equation}\label{subset}
\Big[ L^{p(\cdot), \lambda(\cdot)} \Big]_{L^{p(\cdot),\lambda(\cdot), \theta}} \subset {{\widetilde{L}}^{p(\cdot), \lambda(\cdot), \theta}}.
\end{equation}
Thus, it remains to show that
\begin{equation}\label{embedding-2}
{{\widetilde{L}}^{p(\cdot), \lambda(\cdot), \theta}} \subset \Big[ L^{\infty} \Big]_{L^{p(\cdot),\lambda(\cdot), \theta}}.
\end{equation}

Indeed, let
$f\in {{\widetilde{L}}^{p(\cdot), \lambda(\cdot), \theta}}$.
Take $\eta>0$ arbitrarily. It is enough to find
$g\in L^{\infty}$ such that
$$ \sup_{0<c<p_--1} c^{\frac{\theta}{p_- - c}} \| f-g\|_{L^{p(\cdot)-c, \lambda(\cdot)}} < \eta. $$

First
since
$f\in {{\widetilde{L}}^{p(\cdot), \lambda(\cdot), \theta}}$,
we can choose $\delta_0$ so that
$$ \sup_{0< c< \delta_0} c^{\frac{\theta}{p_- - c}} \| f\|_{ L^{p(\cdot)-c, \lambda(\cdot)} } \leq
\frac{\eta}{2}.
$$
Further, let us take $g \in L^{\infty}$ so that
$|g|+|f-g|=|f|$
and that
$$ c^{\frac{\theta}{p_--c}} \|f- g\|_{ L^{p(\cdot)-c, \lambda(\cdot), \theta}}
\leq C(p, X)(\|f- g\|_{ L^{1} }+\|f-g\|_{L^{p_+}})c^{\frac{\theta}{p_- - c}} \leq \frac{\eta}{2}, \;\;\; $$
$c\in (\delta_0, p_- -1)$,
where the constant $C(p, X)$ depends only on $X$ and $p(\cdot)$.

The latter inequality follows from \eqref{Dien} and \eqref{Dien1} for $r(\cdot)= 1/(p(\cdot)-c)$ and the condition $\lambda(\cdot) \in [0,1]$. Observe here that the constant $A$ in \eqref{Dien} and \eqref{Dien1}
for $r(\cdot)= 1/(p(\cdot)-c)$ does not depend on $c$ because the constant $a$ in \eqref{WL-+} for $r(\cdot)= 1/(p(\cdot)-c)$
can be chosen so that it will be independent of $c$.

Consequently, for all $c\in (0,p_--1)$ and such an $g$ we find that
\begin{align*}
\lefteqn{
c^{\frac{\theta}{p_- - c}} \| f-g\|_{L^{p(\cdot)-c, \lambda(\cdot)}}
}\\
&\leq
\sup_{\tilde{c}\in (0,\delta_0)}\tilde{c}^{\frac{\theta}{p_- -\tilde{c}}} \| f\|_{L^{p(\cdot)-\tilde{c}, \lambda(\cdot)}} +
\sup_{\tilde{c}\in (\delta_0,p_--1)} \tilde{c}^{\frac{\theta}{p_- - \tilde{c}}} \|f-g\|_{L^{p(\cdot)-\tilde{c}, \lambda(\cdot)}}\\
&\leq
\frac{\eta}{2} +
\frac{\eta}{2}= \eta.
\end{align*}

Proposition \ref{Sawano-Result} completes the proof of the theorem.
$\Box$

\section{Duality and Predualily in GVELS on $\sigma$-finite measure spaces}

Now we introduce grand variable exponent Lebesgue spaces
(shortly GVELS) defined on a measure space $(X, \Sigma, \mu)$
with a $\sigma$-finite $\mu$ measure
which is not always finite.
We note that the definition of $L^{p(\cdot)}$
remains unchanged.
\begin{definition}
Let $p(\cdot) \in P(X)$.
Suppose that $a<p_--1$ and $\theta \in {\mathbb R}$.
Then we define
\[ \|f\|_{{\mathcal{L}}^{p(\cdot),\theta, a}}:= \|f\|_{{\mathcal{L}}^{p(\cdot),\theta, a}} \equiv \sup_{0<\kappa \le a} \kappa^{\frac{\theta}{p_- -\kappa}}
\|f\|_{L^{p(\cdot)- \kappa}}. \]
for a $\mu$-measurable function $f$. Here, as before, $\kappa$ is constant. The space ${\mathcal{L}}^{p(\cdot),\kappa, a}$ denotes the set of all $f \in L^1+L^\infty$
for which the norm $\|f\|_{{\mathcal{L}}^{p(\cdot),\theta;a}}$ is finite.
\end{definition}

For the constant exponent $p$
an alternative approach to define grand Lebesgue space was given in \cite{SaUm}.

\begin{remark}
In the previous definition it is not assumed that $X$ is a finite measure space. Here $a$ is a parameter $a \in (0,p_- -1)$.
\end{remark}

We need also the following definition:

\begin{definition}
Let $(X,\Sigma,\mu)$ be a $\sigma$-finite measure space. Let $p(\cdot)$ be a variable exponent satisfying
$1<p_- \le p_+<\infty$,
and let $a<p_--1$ and $\theta \in {\mathbb R}$.
\begin{enumerate}
\item

A $\mu$-measurable
function $b$ is said to be a $(p(\cdot),\theta;a,\kappa)$ block
if there exists $\kappa \in (0,a]$ such that \[
\|b\|_{L^{(p(\cdot)-\kappa)'}} \le \kappa^{\frac{\theta}{p_--\kappa}}.
\]

\item

The space
${\mathcal{H}}^{p(\cdot),\theta;a}$ is the set of all $f \in L^1+L^{\infty}$ for which there exist
$\{\lambda_j\}_{j=1}^\infty \in \ell^1$
and a sequence $\{b_j\}_{j=1}^\infty$ of $(p(\cdot),\theta;a,\kappa_j)$ blocks with $\kappa_j \in (0,a]$ for which
\[ f=\sum_{j=1}^\infty \lambda_j b_j \]
$\mu$-a.e..

The norm of a function
$f \in {\mathcal H}^{p(\cdot),\theta;a}$ is defined to be

\[ \|f\|_{{\mathcal H}^{p(\cdot),\theta;a}} =
\inf\left\{ \sum_{j=1}^\infty|\lambda_j|\,:\, f=\sum_{j=1}^\infty \lambda_j b_j \right\}, \]
where one considers the infimum over all admissible expressions.
\end{enumerate}
\end{definition}

In the sequel we denote by $Y^*$ dual space of $Y$.
We can identify
that dual space of $L^{p(\cdot)}$ with $L^{p'(\cdot)}$ for $p(\cdot) \in P(X)$
 (see e.g., \cite{DHHR,Sawano-Survey}).
\vskip+0.2cm

The following proposition is an immediate consequence of the duality of $L^{p(\cdot)}$ and $L^{p'(\cdot)}$.

\begin{proposition}
Let $(X,\Sigma,\mu)$ be a $\sigma$-finite measure space,
and
let $p(\cdot)\in P(X)$, $a<p_--1$ and $\theta \in {\mathbb R}$.
\begin{enumerate}

\item
We have
${\mathcal{L}}^{p(\cdot),\theta;a}=\Big({\mathcal{H}}^{p(\cdot),\theta;a}\Big)^*$
with equivalence of norms.
In particular,
$f\in {\mathcal{L}}^{p(\cdot),\theta;a}$ induces a bounded linear functional $L_f$
on 	
${\mathcal{H}}^{p(\cdot),\theta;a}$
and 	
$\|f\|_{{\mathcal{L}}^{p(\cdot),\theta;a}} \sim \|L_f\|_{\Big({\mathcal{H}}^{p(\cdot),\theta;a}\Big)^*}$,
if
$f \in {\mathcal{L}}^{p(\cdot),\theta;a}$.
\item
Any bounded linear functional $L$ on ${\mathcal{H}}^{p(\cdot),\theta;a}$
is realized as $L=L_f$ for some $f \in {\mathcal{L}}^{p(\cdot),\theta;a}$.
\end{enumerate}
\end{proposition}

\begin{proof}
Take $f\in {\mathcal{L}}^{p(\cdot), \theta; a}$ and
a $(p(\cdot),\theta;a,\kappa)$ block $b$.
Then there is $\kappa \in (0,a]$ such that
\[
\|b\|_{L^{(p(\cdot)-\kappa)'}} \le \kappa^{\frac{\theta}{p_--\kappa}}.
\]
Hence by H\"older's inequality we have that
\begin{eqnarray*}
\int\limits_X |b(x) f(x)| d\mu(x)
\leq
c_p
\|b\|_{L^{(p(\cdot)-\kappa)'}}\|f \|_{L^{p(\cdot)-\kappa}}
\leq
c_p
\kappa^{\frac{\theta}{p_- -1}} \|f \|_{L^{p(\cdot)-\kappa}}
\leq
c_p
\|f \|_{{\mathcal{L}}^{p(\cdot), \theta; a}}.
\end{eqnarray*}
Hence, ${\mathcal{L}}^{p(\cdot),\theta;a} \subset \Big({\mathcal{H}}^{p(\cdot),\theta;a}\Big)^*$.

To see that ${\mathcal{L}}^{p(\cdot),\theta;a} \supset \Big({\mathcal{H}}^{p(\cdot),\theta;a}\Big)^*$
we take $L \in \Big({\mathcal{H}}^{p(\cdot),\theta;a}\Big)^*$.
Consider the mapping
$\varphi \in L^{(p(\cdot)-\kappa)'} \mapsto L\varphi \in {\mathbb C}$.
Then there exists $g_{\kappa}$ such that
$$ \| g_{\kappa}\|_{L^{p(\cdot)- \kappa}} \leq
C\| L\|_{{\mathcal{H}}^{p(\cdot),\theta;a}} \kappa^{\frac{-\theta}{p_- - \kappa}}
$$
and
$$
L\varphi =
\int\limits_X g_{\kappa}(x) \varphi (x) d\mu(x), \;\;\; \varphi \in L^{(p(\cdot)-\kappa)'}.
$$
Since
\[
\int\limits_X g_{\kappa}(x) \varphi (x) d\mu(x)
=
\int\limits_X g_{\kappa^\dagger}(x) \varphi (x) d\mu(x)
\quad
\varphi \in L^{(p(\cdot)-\kappa)'} \cap L^{(p(\cdot)-\kappa^\dagger)'},
\]
it thus follows that $g_{\kappa}=g_{\kappa^\dagger}$
for $1<\kappa,\kappa^\dagger<p_--1$.
Hence,
by considering $\{g_\kappa\}_{\kappa \in (1,p_--1) \cap {\mathbb Q}}$
we conclude that there is $g\in {\mathcal{L}}^{p(\cdot), \theta; a}$ such that
$$
\| g\|_{{\mathcal{L}}^{p(\cdot), \theta; a}} \leq
C\| L\|_{{\mathcal{H}}^{p(\cdot),\theta;a}}.
$$
\end{proof}

\begin{lemma}
Let $p(\cdot) \in P(X)$, and let $b$ be a $(p(\cdot),\theta;a,\kappa)$ block.
If $l \in {\mathbb N}$ satisfies $2^{-l}a \le \kappa \le 2^{-l+1}a$,
 then there exists a decomposition 	$b=b_1+b_2$ 	
of $\mu$-measurable functions $b_1,b_2$ 	such that
\[ 	\|b_1\|_{L^{(p(\cdot)-2^{-l}a)'}} 	\le A(2^{-l}a)^{\frac{\theta}{p_--2^{-l}a}},
\]
\[	\|b_2\|_{L^{(p(\cdot)-2^{-l+1}a)'}} 	\le A(2^{-l+1}a)^{\frac{\theta}{p_--2^{-l+1}a}}. 	\]

Here $A$ is a universal constant depending on $p(\cdot)$.

%$\mu$-measurable function for which 	 there exists $\kappa \in (0,a)$ such that 	\[
%\|b\|_{L^{(p(\cdot)-\kappa)'}} \le \kappa^{-\frac{\theta}{p_--\kappa}} 	\]

\end{lemma}

\begin{proof}
We note that
\begin{align*}
L^{(p(\cdot)-2^{-l}a)'} \cap L^{(p(\cdot)-2^{-l+1}a)'}
&\approx
\left(L^{p(\cdot)-2^{-l}a}+L^{p(\cdot)-2^{-l+1}a}\right)'\\
&\supset
\Big( L^{p(\cdot)-\kappa}\Big)'\\
&\supset L^{(p(\cdot)-\kappa)'}
\end{align*}
with the embedding constants independent of $l$.
Thus, the result is immediate.
\end{proof}

\begin{corollary}\label{cor}
A function $f \in L^1 + L^\infty$ belongs to ${\mathcal{H}}^{p(\cdot),\theta;a}$
if and only if there exist
$\{\lambda_{2^{-l}a}\}_{l \in{\mathbb N}_0}$ and $\{b_{2^{-l}a}\}_{l \in{\mathbb N}_0}$
for which
\[ f=\sum_{l=0}^\infty \lambda_{2^{-l}a} b_{2^{-l}a} \]
$\mu$-a.e..
If this is the case, then one can arrange that
\[ (2^{-l}a)^{-\frac{\theta}{p_--2^{-l}a}} \|b_{2^{-l}a}\|_{L^{(p(\cdot)-2^{-l}a)'}} \le 1 \]
and that
\[ \sum_{l=0}^\infty |\lambda_{2^{-l}a}| \le A\|f\|_{{\mathcal{H}}^{p(\cdot),\theta;a}} \]
for some positive constant $A$.
\end{corollary}

We omit the detail of the proof of this corollary.
\begin{proposition}
	The space ${\mathcal{H}}^{p(\cdot),\theta;a}$ enjoys with the Fatou property;	if 	$0 \le f_1 \le f_2 \le \cdots \le f_j \le \cdots$ 	is an increasing sequence of ${\mathcal{H}}^{p(\cdot),\theta;a}$, 	which is bounded in ${\mathcal{H}}^{p(\cdot),\theta;a}$, then $\displaystyle \lim_{j \to \infty}f_j \in {\mathcal{H}}^{p(\cdot),\theta;a}$ 	and $\displaystyle \left\|\lim_{j \to \infty}f_j\right\|_{{\mathcal{H}}^{p(\cdot),\theta;a}} 	 =\lim_{j \to \infty}\left\|f_j\right\|_{{\mathcal{H}}^{p(\cdot),\theta;a}}$.
\end{proposition}

%%%%%%%%%%%%%%%%%%%%%%%%

\begin{proof}
We follow the idea of \cite{ISY14}.
We may assume
$\|f_j\|_{{\mathcal H}^{p(\cdot),\theta;a}}<1$.
By Corollary \ref{cor} we can find an expression of each $f_j$;
\[ f_j=\sum_{l=0}^\infty \lambda^j_{2^{-l}a} b^j_{2^{-l}a} \]
$\mu$-a.e.,
where
\[ (2^{-l}a)^{\frac{\theta}{p_--2^{-l}a}} \|b^j_{2^{-l}a}\|_{L^{(p(\cdot)-2^{-l}a)'}} \le 1 \]
and that
\[ \sum_{l=0}^\infty |\lambda^j_{2^{-l}a}| \le A\|f_j\|_{{\mathcal{H}}^{p(\cdot),\theta;a}}. \]
Thanks to the weak compactness of $L^{(p(\cdot)-2^{-l}a)'}$ for $l \in {\mathbb N}_0$,
 after we pass to a subsequence,
we can assume that
\[ \lim_{j \to \infty}b^j_{2^{-l}a}=b_{2^{-l}a} \]
exists in the weak topology of $L^{(p(\cdot)-{2^{-l}a})'}$ for each $l \in {\mathbb N}_0$.
Using the Weierstrass theorem, we can assume that
\[ \lim_{j \to \infty}\lambda^j_{2^{-l}a}=\lambda_{2^{-l}a} \]
exist.
We set
\[ f\equiv \sum_{l=0}^\infty \lambda_{2^{-l}a} b_{2^{-l}a} \in {\mathcal{H}}^{p(\cdot),\theta, a}. \]

To prove $\displaystyle \lim_{j \to \infty}f_j \in {\mathcal{H}}^{p(\cdot),\theta;a}$ and $\displaystyle \left\|\lim_{j \to \infty}f_j\right\|_{ {\mathcal{H}}^{p(\cdot),\theta;a}} =\lim_{j \to \infty}\left\|f_j\right\|_{{\mathcal{H}}^{p(\cdot),\theta;a}}$,
it suffices to show that
\[ f=\lim_{j \to \infty}f_j \]
almost everywhere, because
\[ \liminf_{j \to \infty} \sum_{l=0}^\infty |\lambda^j_{2^{-l}a}| \ge \sum_{l=0}^\infty |\lambda_{2^{-l}a}| \]
by the Fatou theorem. To this end, it suffices to show that
\[ \int\limits_G f(x)\,d\mu(x)= \lim_{j \to \infty} \int\limits_G f_j(x)\,d\mu(x) \]
for any measurable set $G$ with finite $\mu$-measure,
or equivalently,
\begin{equation}\label{sum}
\sum_{l=0}^\infty \lambda^j_{2^{-l}a} \int_{G}b^j_{2^{-l}a}(x)\,d\mu(x) \to
\sum_{l=0}^\infty \lambda_{2^{-l}a} \int_{G}b_{2^{-l}a}(x)\,d\mu(x),
\end{equation}
since we know that
$\chi_G \in L^{p_+} \cap L^{p_-} \subset {\mathcal{L}}^{p(\cdot),\theta;a}$.

We also know that
\[ \left|\int_{G}b^j_{2^{-l}a}(x)\,d\mu(x)\right| \lesssim 2^{-l\delta} \]
for some $\delta>0$ by the H\"{o}lder inequality with the constants independent of $j$ and $l$ and
\[ \lambda^j_{2^{-l}a} \int_{G}b^j_{2^{-l}a}(x)\,d\mu(x) \to \lambda_{2^{-l}a} \int_{G}b_{2^{-l}a}(x)\,d\mu(x). \]
Thus, we are in the position of applying the Lebesgue convergence theorem
to see that \eqref{sum} holds.
%\[ \sum_{l=0}^\infty \lambda^j_{2^{-l}a} \int_{G}b^j_{2^{-l}a}(x)\,d\mu(x) \to
%\sum_{l=0}^\infty \lambda_{2^{-l}a} \int_{G}b_{2^{-l}a}(x)\,d\mu(x) \]
%as we wished.

\end{proof}

We define
$[L^{p(\cdot)}]_{{\mathcal{L}}^{p(\cdot),\theta;a}}$ which is the closure of $L^{p(\cdot)} \cap L^{p_--a}$
in ${\mathcal{L}}^{p(\cdot),\theta;a}$.
The intersection with $L^{p_--a}$
is taken because the space is defined on a set which might be of infinite measure.
This operation is necessary because
$L^{p(\cdot)}$ is not always included in
$[L^{p(\cdot)}]_{{\mathcal{L}}^{p(\cdot),\theta;a}}$.
In fact in the case of constant exponent,
$L^{p}$ is included in
${{\mathcal{L}}^{p(\cdot),\theta;a}}$
if and only if the underlying measure is finite.
Now we will see that $g$ induces a linear functional $K_g$
on ${\mathcal{L}}^{p(\cdot),\theta;a}$ and $\|g\|_{{\mathcal{H}}^{p(\cdot),\theta;a}(X,\mu)}
\sim \|K_g\|_{({\mathcal{L}}^{p(\cdot),\theta;a})^*}$, if $g \in {\mathcal{H}}^{p(\cdot),\theta;a}$.
In the case $\mu(X)< \infty$ we will assume that
$[L^{p(\cdot)}]_{{\mathcal{L}}^{p(\cdot), \theta, a}}$ is the closure of $L^{p(\cdot)}$ in
${\mathcal L}^{p(\cdot), \theta, a}$.

\begin{theorem}
Suppose that $\mu$ is $\sigma$- finite.
Let $p(\cdot)$ be a variable exponent satisfying
$1<p_- \le p_+<\infty$, and let $0<a<p_--1$ and $\theta \in {\mathbb R}$. 	With norm equivalence, we have
$\Big( [L^{p(\cdot)}]_{{\mathcal{L}}^{p(\cdot),\theta;a}}\Big)^*={\mathcal{H}}^{p(\cdot),\theta;a}$. More precisely,
	any bounded linear functional $K$ on $[L^{p(\cdot)}]_{{\mathcal{L}}^{p(\cdot),\theta;a}}$
	is realized as $K=K_g$
for some $g \in {\mathcal{H}}^{p(\cdot),\theta;a}$.

\end{theorem}

\begin{proof}
We prove the theorem for $\mu(X)=\infty$. The proof for $\mu(X)< \infty$ is similar.
Let $K$ be a bounded linear functional on $[L^{p(\cdot),\theta}]_{{\mathcal{L}}^{p(\cdot),\theta;a}}$. Then
the inclusion mapping
$\iota:L^{p(\cdot)} \cap L^{p_--a} \hookrightarrow [L^{p(\cdot)}]_{{\mathcal{L}}^{p(\cdot),\theta;a}}$
induces a linear functional $K \in \Big(L^{p(\cdot)}\Big)^*$,
so that there exists $g \in L^{p'(\cdot)}+L^{(p_--a)'}$ such that
\begin{equation}\label{eq:170704-1}
K(f)=\int_X f(x)g(x)\,d\mu(x), \;\;\; f\in L^{p(\cdot)},
\end{equation}
We need to show that $g \in {\mathcal{H}}^{p(\cdot),\theta;a}$.
We define
\[ K_j(f)=\int_{X_j} f(x)\chi_{\{|g| \le j\}}(x)g(x)\,d\mu(x), \;\;\; f\in L^{p(\cdot)}, \]
where $\{X_j\}$ is an increasing sequence of finite measure sets.
Since $\mu(X_j)<\infty$,
we have $\chi_{\{|g| \le j\} \cap X_j}g \in {\mathcal{H}}^{p(\cdot),\theta;a}(X,\mu)$
and the norm is bounded by a constant multiple of $\|K\|$
due to the Hahn-Banach theorem and the duality
${\mathcal{H}}^{p(\cdot),\theta;a}$--${\mathcal{L}}^{p(\cdot),\theta;a}$.
By the Fatou property, we conclude that $g\in {\mathcal{H}}^{p(\cdot),\theta;a}$.
Thus, $K=K_g$ since we can now assume that $f \in [L^{p(\cdot)}]_{{\mathcal{L}}^{p(\cdot),\theta;a}}$ in (\ref{eq:170704-1}).
\end{proof}

\section*{Acknowledgement}

Yoshihro Sawano is partially supported by Grant-in-Aid for Scientific Research (C),
No. 16K05209, Japan Society for
the Promotion of Science.

\end{document}